\documentclass[12pt]{amsart}

\usepackage{amsmath}
\usepackage{amssymb}
\usepackage{amscd}

\def\NZQ{\mathbb}               
\def\NN{{\NZQ N}}
\def\QQ{{\NZQ Q}}
\def\ZZ{{\NZQ Z}}
\def\RR{{\NZQ R}}

\def\PP{{\NZQ P}}

\newtheorem{Theorem}{Theorem}[section]
\newtheorem{Lemma}[Theorem]{Lemma}
\newtheorem{Corollary}[Theorem]{Corollary}
\newtheorem{Proposition}[Theorem]{Proposition}
\newtheorem{Remark}[Theorem]{Remark}

\newtheorem{Example}[Theorem]{Example}

\newtheorem{Definition}[Theorem]{Definition}

\let\epsilon\varepsilon
\let\phi=\varphi
\let\kappa=\varkappa

\begin{document}

\title{The Proj of the Rees algebra of a graded family of ideals}

\dedicatory{Dedicated to the memory of J\"urgen Herzog}

\author{Steven Dale Cutkosky}

\thanks{The  author was partially supported by NSF grant DMS-2348849.}

\address{Steven Dale Cutkosky, Department of Mathematics,
University of Missouri, Columbia, MO 65211, USA}
\email{cutkoskys@missouri.edu}

\begin{abstract} In this article we investigate the condition that the Proj of a  Rees algebra of a graded family of ideals in a Noetherian local ring $R$ is Noetherian. In many cases, the Proj will be Noetherian even when the  Rees algebra is not. For instance, the Proj of the Rees algebra of a graded filtration of ideals will alway be Noetherian if the analytic spread of the filtration is zero.

The Proj of a Rees algebra of a divisorial filtration on a two dimensional normal excellent local ring is always Noetherian, as was proven by Russo and later with a different proof by the author.  
We give examples in this paper of  divisorial filtrations on three dimensional normal excellent local rings whose Proj is not Noetherian, 
showing that this theorem does not extend to higher dimensions.  

A consequence of the fact that the Proj of a  divisorial filtration over a two dimensional excellent normal local ring is always Noetherian is that the preimage of the maximal ideal of $R$ in the Proj has only finitely many irreducible components. As a consequence, 
the fiber cone of such a filtration  has only finitely many minimal primes. We give an example of a graded filtration of ideals in a two dimensional regular local ring such that the preimage of the maximal ideal in the Proj of the Rees algebra of the filtration has infinitely many irreducible components, so that the Proj is not Noetherian, and the fiber cone of the filtration has infinitely many minimal primes. 
\end{abstract}


\maketitle

\section{Introduction} If $A$ is a ring, then a graded family of ideals in $A$ is a family $\mathcal I=\{I_n\}_{n\in \NN}$ of ideals such that $I_0=A$ and $I_nI_n\subset I_{m+n}$ for all $m,n$.

Let $(R,m_R)$ be a $d$-dimensional normal, Noetherian local domain, and $\mathcal I=\{I_n\}$ be a graded family of $m_R$-primary ideals on $R$. The Rees algebra $R[\mathcal I]=\oplus_{n\ge 0}I_n$ of $\mathcal I$ is often not Noetherian. The simplest example of  a non Noetherian  graded family is given by  $I_n=m_R$ for all $n$. 

Let $\pi:\mbox{Proj}(R[\mathcal I])\rightarrow \mbox{Spec}(R)$ be the natural projection. We have that 
$$
\mbox{Proj}(R[\mathcal I])\setminus \pi^{-1}(m_R)\rightarrow \mbox{Spec}(R)\setminus \{m_R\}
$$
 is an isomorphism, so everything interesting is happening in the fiber over the maximal ideal. 

The analytic spread $\ell(\mathcal I)$ of $\mathcal I$ is defined as
$$
\ell(\mathcal I)=\dim R[\mathcal I]/m_RR[\mathcal I].
$$
 We have that 
 $$
 0\le \ell(\mathcal I)\le d=\dim R,
 $$
 as shown in Lemma 3.6 \cite{CS}, and 
$\ell(\mathcal I)$ can be any number from 0 to $d$. In the  case that $\ell(\mathcal I)=0$, $\pi^{-1}(m_R)=\emptyset$, so $\mbox{Proj}(R[\mathcal I])\cong \mbox{Spec}(R)\setminus \{m_R\}$, which is a Noetherian scheme. The filtration $\mathcal I$ defined by $I_n=m_R$ for all $n$ gives a simple example where $\ell(\mathcal I)=0$.

Divisorial filtrations are defined in Section \ref{SecVal}.
For divisorial filtrations on an excellent normal local domain of dimension two, we have the following theorem.

\begin{Theorem}\label{Theorem1}(Theorem 4 \cite{R}, Theorem 1.3 \cite{C1}) Suppose that $R$ is a 2-dimensional normal excellent local ring and $\mathcal I$ is a divisorial filtration on $R$.  Then $\mbox{Proj}(R[\mathcal I])$ is a Noetherian normal scheme. It is proper over $\mbox{Spec}(R)$ if and only if $R[\mathcal I]$ is Noetherian.
\end{Theorem}

If $R$ is normal and complete of dimension 2 with algebraically closed residue field of characteristic zero, then $R[\mathcal I]$ is Noetherian for all divisorial filtrations $\mathcal I$ of $R$
 if and only if $R$ has a rational singularity. This is shown in Proposition \ref{IntProp'}.

We give an example in Theorem \ref{Theorem2} showing that $\mbox{Proj}(R[\mathcal I])$ can be non Noetherian when $\mathcal I$ is a divisorial filtration on a normal excellent local ring $R$ of dimension 3, so Theorem \ref{Theorem1} does not extend to higher dimensions. 

A consequence of Theorem \ref{Theorem1} is that the fiber cone $R[\mathcal I]/m_RR[\mathcal I]$ has only finitely many minimal primes when $R$ is normal and excellent of dimension 2 and $\mathcal I$ is a divisorial filtration. In Theorem \ref{Theorem3} we give an example of a graded family of $m_R$-primary ideals on a 2-dimensional regular local ring such that $R[\mathcal I]/m_RR[\mathcal I]$ is a 2-dimensional ring which has infinitely many minimal primes. Necessarily, $\mbox{Proj}(R[\mathcal I])$ is not Noetherian. 

A remaining question is if there exists a divisorial filtration $\mathcal I$ on a normal, excellent local ring $R$ such that $R[\mathcal I]/m_RR[\mathcal I]$ has infinitely many minimal primes. 

\section{Divisorial valuations and divisorial filtrations}\label{SecVal}

Suppose that $v$ is a valuation of a field $K$. Write  $\mathcal O_v$ for the valuation ring of $v$ with maximal ideal $m_v$, $vK$ for the value group of $v$ and $Kv$ for the residue field of $\mathcal O_v$.
A valuation $v$ of the quotient field $K$ of a local domain $R$ dominates $R$ if $R\subset \mathcal O_v$ and $m_v\cap R=m_R$.

If $v$ is a valuation of the quotient field of $R$ which is nonnegative on $R$ and $\pi:X\rightarrow \mbox{Spec}(R)$ is a proper birational morphism then the center of $v$ on $X$ is the unique (not necessarily closed) point $z\in X$ such that $\mathcal O_v$ dominates $\mathcal O_{X,z}$.

\begin{Definition}\label{DivvalDef}(Definition 9.3.1 \cite{HS}) Let $R$ be a Noetherian integral domain with quotient field $K$. Let $v$ be a valuation of $K$ such that $R\subset \mathcal O_v$. Let $P=m_v\cap R$. If $\mbox{trdeg}_{(R/P)_P}Kv=\mbox{ht}(P)-1$, then $v$ is said to be a divisorial valuation of $R$.
A divisorial valuation $v$ is an $m_R$-valuation if $v$ dominates $R$.
\end{Definition}

 If $R$ is an excellent local domain, then the divisorial valuations are the valuation rings of the quotient field of $R$ which are essentially of finite type over $R$.
 
  The value group of a divisorial valuation is $\ZZ$ (Theorem 9.3.2 and Proposition 6.3.4 \cite{HS}). 
A divisorial filtration of a domain $R$ is a filtration $\mathcal I=\{I_n\}$ where
$$
I_n=I(v_1)_{na_1}\cap I(v_2)_{na_2}\cap\cdots\cap I(v_r)_{na_r}
$$
for $n\ge 0$, where $v_1,\ldots,v_r$ are discrete valuations which are nonnegative on $R$, $a_1,\ldots,a_r$ are nonnegative rational numbers and
$$
I(v_i)_{\lambda}=\{f\in R\mid v_i(f)\ge \lambda\}
$$
for $\lambda$ a nonnegative real number.

\section{Surface singularities} 

We will use the following fact in this section.
A Rees algebra $R[\mathcal I]=\oplus_{n\ge 0}I_n$ is Noetherian if and only if $R[\mathcal I]$ is a finitely generated $R$-algebra, by Corollary 4 on page 184 \cite{B}.
 
 In this section we suppose that $R$ is a complete normal local domain of dimension 2, unless explicitly stated otherwise.

 \begin{Definition}(\cite{MS}, \cite{C2}) Let $\Omega(R)$ be the set of $m_R$-valuations. $R$ satisfies condition N if for any $v\in \Omega(R)$, there exists an $m_R$-primary ideal $A$ in $R$ such that if $X_A\rightarrow \mbox{Spec}(R)$ is the normalization of the blowup of $A$, the center of $v$ on $X_A$ is the reduced closed fiber of the map. 
 \end{Definition}

  \begin{Proposition}\label{PropFGN}(Proposition on page 60 of \cite{C2}) Suppose that v is an $m_R$-valuation.  Let 
 $$
 I(v)_n=\{f\in R\mid v(f)\ge n\}.
 $$
Then $\oplus_{n\ge 0}I(v)_n$ is a finitely generated $R$-algebra if and only if 
 there exists
an $m_R$-primary ideal $J\subset R$ such that the blow up $X$ of $J$ is normal, and the center
of $v$ on $X$ is the reduced exceptional fiber. 
\end{Proposition}

 Let $\pi:Y\rightarrow \mbox{Spec}(R)$ be a resolution of singularities. Let $C_1,\ldots,C_r$ be the integral exceptional divisors of $\pi$. Let $E_{Y^+}$ be the set of divisors $D=\sum a_iC_i$ such that $a_i\in \ZZ$ and $D$ is nef; that is, $(D\cdot C_i)\le 0$ for all $1\le i\le r$. Let $E_{Y^{++}}\subset E_{Y^+}$ be the divisors $D$ such that $\mathcal O_Y(-D)$ is generated by global sections.
 
 \begin{Theorem}\label{TheoremNG}(Theorem 3.10 \cite{Go}, \cite{C2}) $R$ satisfies condition N if and only if for each desingularization $\pi:Y\rightarrow \mbox{Spec}(R)$, there exists a positive integer $n$ such that $nE_{Y^+}\subset E_{Y^{++}}$.
 \end{Theorem}
 
 \begin{Theorem}\label{TheoremN4}(Theorem 4 \cite{C2}) $R$ satisfies condition N if and only if the divisor class group $\mbox{Cl}(R)$ of $R$ is a torsion group.
 \end{Theorem}

 \begin{Lemma}\label{LemmaGN} Suppose that the divisor class group $\mbox{Cl}(R)$ of $R$ is a torsion group. Let $\pi:X\rightarrow\mbox{Spec}(R)$ be a resolution of singularities, with integral exceptional divisors $C_1,\ldots, C_r$. Suppose that $F\subset X$ is an integral curve such that $F$ does not have exceptional support.  Then there exists $n_0\in \ZZ$ and $a_1,\ldots, a_r\in \ZZ_{\ge 0}$ such that
 $$
 \mathcal O_X(-n_0F)\cong\mathcal O_X(a_1C_1+\cdots+a_rC_r).
 $$
 \end{Lemma}
 
 \begin{proof} Let $\pi(F)=C$, an integral curve on $\mbox{Spec}(R)$. For $n\in \ZZ_{>0}$, 
 $$
 \pi_*\mathcal O_X(-nF)=\mathcal O_{\mbox{Spec}(R)}(-nC),
 $$
  so there exists $n_0>0$ such that
 $$
 \pi_*\mathcal O_X(-n_0F)=\mathcal O_{\mbox{Spec}(R)}(-n_0C) \cong \mathcal O_{\mbox{Spec}(R)}\cong \tilde{R},
 $$
 the sheaf associated to $R$,  since $\mbox{Cl}(R)$ is a torsion group. Thus there exists $g\in R$ such that 
  $$
  \pi_*\mathcal O_X(-n_0F)=gR.
  $$
   Since $X\setminus \pi^{-1}(m_R)\rightarrow \mbox{Spec}(R)\setminus \{m_R\}$ is an isomorphism, the divisor of $g$ on $X$ is 
  $$
  (g)_X=n_0F+\sum_{i=1}^ra_iC_i
  $$
  for some $a_i\in \ZZ_{\ge 0}$. Thus  
  $\mathcal O_X(-n_0F)\cong \mathcal O_X(\sum_{i=1}^ra_iC_i)$.
 \end{proof}

 \begin{Lemma}\label{Lemmalast} Suppose that $R$ is an excellent normal local domain  and $\pi:X\rightarrow\mbox{Spec}(R)$ is the blowup of an $m_R$-primary ideal of $R$. Let $C_1,\ldots,C_r$ be the integral exceptional divisors of $\pi$. Suppose that $u_1,\ldots,u_r\in \ZZ$. Then 
 $$
 \Gamma(X,\mathcal O_X(u_1C_1+\ldots+u_rC_r))=\Gamma(X,\mathcal O_X(d_1C_1+\cdots+d_rC_r))
 $$
 where $d_i=\min\{u_i,0\}$ for $1\le i\le r$.
 \end{Lemma}
 
 \begin{proof}
 
 $$
 \Gamma(X,\mathcal O_X(\sum_{i=1}^ru_iC_i))\subset \Gamma(X,\mathcal O_X(\sum_{i=1}^ru_iC_i))_P\cong R_P
 $$
 for all height one prime ideals $P$ of $R$
 since $X\setminus \pi^{-1}(m_R)= \mbox{Spec}(R)\setminus\{m_R\}$. Now $R=\cap R_P$ where the intersection is over the height one prime ideals $P$ of $R$ since $R$ is normal. Thus $\Gamma(X,\mathcal O_X(\sum_{i=1}^ru_iC_i))\subset R$. Let $v_i$ be the valuation of the valuation ring $\mathcal O_{X,C_i}$. Then $v_i$ dominates $R$ so that $v_i(f)\ge 0$ for all $f\in R$.
 
 \end{proof}

 \begin{Lemma}\label{LemmaAN} Suppose that $R$ is an excellent local domain and $\pi:X\rightarrow\mbox{Spec}(R)$ is the blowup of an ideal of $R$. Suppose that $\mathcal L$ is an invertible sheaf on $X$ which is generated by global sections. Then
 $\oplus_{n\ge 0}\Gamma(X,\mathcal L^n)$ is a finitely generated $R$-algebra.
 \end{Lemma}
 
 \begin{proof} Let $\mathcal O_X^r\rightarrow \mathcal L$ be a surjection, which exists since $\mathcal L$ is generated by global sections. By Theorem II.7.1 \cite{H}, there exists a unique $R$-morphism $\phi:X\rightarrow Z:=\PP^{r-1}_R$ such that $\phi^*(\mathcal O_Z(1))\cong \mathcal L$. The map $\phi$ is a projective $R$-morphism since $\pi$ is (Proposition II.5.5.5 \cite{EGAII}), so $\phi_*\mathcal O_X$ is a coherent $\mathcal O_Z$-module by Theorem II.8.8 \cite{H}. Thus $Y:=\phi(X)$ is a closed integral subscheme of $Z$, with structure sheaf 
 $ \phi_*\mathcal O_X$, and $\phi(X)=Y=\mbox{Proj}(S)$ where $S$ is a finitely generated graded $R$-algebra which is generated in degree 1 and is a domain. We have that
 $$
 \mathcal O_Y(n)\cong (\phi_*\mathcal O_X)(n)\cong \phi_*(\mathcal L^n)
 $$
 by the projection formula. By the proof of Theorem II.5.19 \cite{H} (and Remark II.5.19.2 \cite{H}),
 $$
 \oplus_{n\ge 0}\Gamma(X,\mathcal L^n)\cong \oplus_{n\ge 0}\Gamma(Y,\mathcal O_Y(n))
 $$
 is a finite $S$-module, so $\oplus_{n\ge 0}\Gamma(X,\mathcal L^n)$ is a finitely generated $R$-algebra.
 \end{proof}

 \begin{Proposition}\label{PropBN} Suppose that $\mbox{Cl}(R)$ is a torsion group and $\mathcal I=\{I_n\}$ is a divisorial filtration of $R$. Then $\oplus_{n\ge 0}I_n$ is a finitely generated $R$-algebra. 
 \end{Proposition}
 
 \begin{proof} There exists a resolution of singularities $\pi:X\rightarrow\mbox{Spec}(R)$ with integral exceptional divisors $C_1,\ldots,C_r$ and integral divisors $F_1,\ldots,F_s$ which do not contract to $m_R$, such that 
 $$
 I_n=\Gamma(X,\mathcal O_X(-\lceil nD\rceil))
 $$
 where $D=\sum a_iC_i+\sum b_jF_j$ is an effective $\QQ$-divisor on $X$, and 
 $$
 \lceil nD\rceil =\sum \lceil na_i\rceil C_i+\sum \lceil nb_j\rceil F_j
 $$
 where $\lceil x\rceil$ is the roundup of the real number $x$. By Lemma \ref{LemmaGN}, there exists $n_0>0$ such that $n_0D$ is an integral divisor and 
 $\mathcal O_X(-n_0D)\cong \mathcal O_X(\sum_{i=1}^r u_iC_i)$ for some $u_i\in \ZZ$.  By Lemma \ref{Lemmalast}, for all $n\in \ZZ_{>0}$,
 $$
 \mathcal O_X(-nn_0D)\cong\mathcal O_X(\sum_{i=1}^rnd_iC_i)
 $$
 where $d_i=\min\{u_i,0\}$ for $1\le i\le r$. Let $D'=\sum_{i=1}^r-d_iC_i$, an effective divisor.

 There then exists an effective  $\QQ$-divisor $\Delta$ on $X$ such that $-\Delta$ is nef ($(\Delta\cdot F)\le 0$ for all exceptional curves $F$ of $\phi$) and
 $$
 \Gamma(X,\mathcal O_X(-\lceil nD'\rceil))=\Gamma(X,\mathcal O_X(-\lceil n\Delta\rceil))
 $$
 for all $n\ge 0$
 by relative Zariski decomposition (Lemma 4.1 and Lemma 4.3 \cite{C1}). 
 
 Let $m_0$ be such that $m_0\Delta$ is an (integral) Cartier divisor. Then there exists a positive integer $s_0$ such that $\mathcal O_X(-s_0m_0\Delta)$ is generated by global sections by Theorems \ref{TheoremNG} and  \ref{TheoremN4}, since $\mbox{Cl}(R)$ is a torsion group.  
 Thus $\oplus_{n\ge 0}\Gamma(X,\mathcal O_X(-ns_0m_0\Delta))$ is a finitely generated $R$-algebra by Lemma \ref{LemmaAN}.  Thus
 $\oplus_{n\ge 0}I_{nm_0n_0}$ is a finitely generated $R$-algebra since
 $$
 \oplus_{n\ge 0}I_{nm_0n_0s_0}\cong\oplus_{n\ge 0}\Gamma(X,\mathcal O_X(-nm_0s_0D'))\cong \oplus_{n\ge 0}\Gamma(X,\mathcal O_X(-nm_0s_0\Delta)).
 $$
 
 Suppose that $f\in I_n$ for some $n$. Then $f^{m_0n_0s_0}\in I_{nm_0n_0s_0}$, so 
 $\oplus_{n\ge 0}I_n$ is integral over $\oplus_{n\ge 0} I_{nm_0n_0s_0}$. Thus $\oplus_{n\ge 0}I_n$ is a finitely generated $R$-algebra by Scholie IV.7.8.3 \cite{EGAIV} since 
 $\oplus_{n\ge 0}I_{nm_0n_0s_0}$
  is a finitely generated $R$-algebra, and hence is an excellent domain.
 \end{proof}

\begin{Proposition}\label{SurfSing} Suppose that $R$ is a complete normal local domain of dimension two. Then the following are equivalent
\begin{enumerate}
\item[1)] 
$\oplus_{n\ge 0}I_v(n)$ is Noetherian for all $m_R$-valuations $v$ of $R$.
\item[2)] $\oplus_{n\ge 0}I_v(n)/I_v(n+1)$
 is Noetherian for all $m_R$-valuations $v$ of  $R$.
 \item[3)] $\mbox{Cl}(R)$ is a torsion group.
  \end{enumerate}
\end{Proposition}

\begin{proof} 

1) equivalent to 3) is by Proposition \ref{PropFGN} and Theorem \ref{TheoremN4}.


Finally, we observe that  $\oplus_{n\ge 0}I_v(n)$ is Noetherian if and only if $\sum_{n\ge 0}I_v(n)/I_v(n+1)$ is Noetherian. If $\oplus_{n\ge 0}I_v(n)/I_v(n+1)$ is Noetherian, then this follows from Corollary 1 to Proposition 12 of III.2.9 (page 181) \cite{B}. If $R[\mathcal I]=\oplus_{n\ge 0}I_v(n)\cong \sum_{n\ge 0}I_v(n)t^n$ is Noetherian, then so is the extended Rees Algebra $R[\mathcal I,t^{-1}]$. Thus 
$$
\oplus_{n\ge 0}I_v(n)/I_v(n+1)\cong R[\mathcal I,t^{-1}]/t^{-1}R[\mathcal I,t^{-1}]
$$
 is Noetherian. 

\end{proof}

\begin{Proposition}\label{IntProp} Suppose that $R$ is a complete normal local domain of dimension 2. Then $R[\mathcal I]$ is Noetherian for all divisorial filtrations  of $R$ if and only if the divisor class group $\mbox{Cl}(R)$ is a torsion group.
\end{Proposition}

\begin{proof} If $\mbox{Cl}(R)$ is a torsion group, then $R[\mathcal I]$ is Noetherian for all divisorial filtrations $\mathcal I$ of $R$ by Proposition \ref{PropBN}. If $R[\mathcal I]$ is Noetherian for all divisorial filtrations $\mathcal I$ of $R$ then certainly $R[\mathcal I]$ is Noetherian for all filtrations $\mathcal I=\{I_v(n)\}$ where $v$ is a divisorial valuation which dominates $R$. Thus $\mbox{Cl}(R)$ is a torsion group  by Proposition \ref{SurfSing}.
\end{proof}

The following propositions follow from Propositions \ref{SurfSing} and \ref{IntProp} since if $R$ is a complete normal local domain of dimension two whose residue field is algebraically closed of characteristic zero, then $\mbox{Cl}(R)$ is a torsion group if and only if $R$ has a rational singularity by Propositions 17.1 and 17.3  \cite{L}.

\begin{Proposition}\label{SurfSing'} Suppose that $R$ is a complete normal local domain of dimension two whose residue field is algebraically closed of characteristic zero. Then the following are equivalent
\begin{enumerate}
\item[1)] 
$\oplus_{n\ge 0}I_v(n)$ is Noetherian for all $m_R$-valuations $v$ of $R$.
\item[2)] $\oplus_{n\ge 0}I_v(n)/I_v(n+1)$
 is Noetherian for all $m_R$-valuations $v$ of $R$.
 \item[3)] $R$ has a rational singularity.
  \end{enumerate}
\end{Proposition}

\begin{Proposition}\label{IntProp'} Suppose that $R$ is a complete normal local domain of dimension 2 whose residue field is algebraically closed of characteristic zero. Then $R[\mathcal I]$ is Noetherian for all divisorial filtrations  of $R$ if and only if $R$ has a rational singularity.
\end{Proposition}

If $k$ is algebraically closed of positive characteristic, then there exist non rational two dimensional singularities $R$ over $k$ with torsion divisor class groups showing that the above two propositions are not true in positive characteristic. A simple example is given on page 428 of \cite{C0}. Let $k$ be algebraically closed of positive characteristic. The simple elliptic singularity  
$$
R=k[[x,y,z]]/(z^2-y^3-x^7)
$$
 of \cite{La} is a nonrational singularity, with $\mbox{Cl}(R)\cong k^+$, which is a torsion group (in positive characteristic).

\section{A divisorial filtration $\mathcal I$ on a 3 dimensional normal excellent local ring such that $\mbox{Proj}(R[\mathcal I])$ is not Noetherian.}\label{SecDiv}

\begin{Lemma}\label{GaussDivVal} Let $R$ be a Noetherian local domain with quotient field $K$ and let $v$ be a valuation of $K$ which dominates $R$ which is a divisorial valuation with respect to $R$. Let $K(z)$ be a rational function field over $K$ and and $w$ be the Gauss valuation of $K(z)$ defined by 
$$
w(f(z))=\min\{v(a_i)+i\mid a_i\ne 0\}
$$
if $0\ne f(z)=a_0+a_1z+\cdots+a_dz^d\in K[z]$ with $a_i\in K$. Then $w$ is a divisorial valuation with respect to $A=R[z]$ with center at the maximal ideal $\mathfrak m=m_RA+zA$ of $A$.
\end{Lemma}

\begin{proof} We have that $A\subset \mathcal O_{w}$ and $m_{w}\cap A=\mathfrak m$. There exists $b\in \mathcal O_{v}$ such that $v(b)=1$. Let $x=\frac{z}{b}\in \mathcal O_{\omega}$. Suppose that the residue $\overline{x}$ of $x$ in $K(z)w$ is algebraic over $Kv$. Then there exists $r\in \ZZ_{>0}$ and  $a_0,a_1,\ldots,a_{r-1}\in \mathcal O_v$ such that
\begin{equation}\label{eq1}
\overline{a_0}+\overline{a_1}\,\overline x+\cdots+\overline{a_{r-1}}\,\overline x^{r-1}+\overline x^r=0
\end{equation}
where $\overline a_i$ is the residue of $a_i$ in $Kv$. By equation (\ref{eq1}) we have that
$$
w(a_0b^r+a_1b^{r-1}z+\cdots+a_{r-1}bz^{r-1}+z^r)>rv(b).
$$
But $v(a_ib^{r-i})\ge r-i$ so 
$w(a_ib^{r-i}z^i)\ge rv(b)$ for all $i$, so that 
$$
w(a_0b^r+a_1b^{r-1}z+\cdots+a_{r-1}bz^{r-1}+z^r)=w(z^r)=r
$$
giving a contradiction. Thus $\mbox{trdeg}_{Kv}K(z)w\ge 1$. Since $A/\mathfrak m=R/m_R$,
$$
\mbox{trdeg}_{A/\mathfrak m}K(z)w=\mbox{trdeg}_{R/m_R}Kv+\mbox{trdeg}_{Kv}K(z)w\ge \dim A-1.
$$
Now $\mbox{trdeg}_{A/\mathfrak m}K(z)w\le \mbox{ht}(\mathfrak m)-1=\dim A-1$, by the dimension inequality. Thus $\mbox{trdeg}_{A/\mathfrak m}K(z)w= \dim A-1=\mbox{ht}(\mathfrak m)-1$.
\end{proof}

For $n\in \NN$, let $I_v(n)=\{f\in R\mid v(f)\ge n\}$, and 
$I_w(n)=\{f\in A\mid w(f)\ge n\}$. Then 
$$
I_w(n)=\sum_{i+j=n}I_v(i)z^jA.
$$
Let
$J_w(n)=I_w(n)A_{\mathfrak m}=\{f\in A_{\mathfrak m}\mid w(f)\ge n\}$. Let $\mathcal I=\{I_w(n)\}$ be the graded filtration of $A$, with Rees algebra
$$
A[\mathcal I]=\sum_{n\ge 0}I_w(n)t^n\subset A[t],
$$
where $A[t]$ is a polynomial ring over $A$. Let $\mathcal J=\{J_{\omega}(n)\}$ be the graded filtration of $A_{\mathfrak m}$ with Rees algebra 
$$
A_{\mathfrak m}[\mathcal J]=\sum_{n\ge 0} J_w(n)t^n\subset A_{\mathfrak m}[t].
$$
We have that $A_{\mathfrak m}[\mathcal J]=A[\mathcal I]\otimes_AA_{\mathfrak m}$ and  $A_{\mathfrak m}[\mathcal I]_{(zt)}=A[\mathcal J]_{(zt)}\otimes_AA_{\mathfrak m}$.

\begin{Proposition}\label{PropEx1}  Let notation be as in Lemma \ref{GaussDivVal}. Suppose that the associated graded ring  $\oplus_{n\ge 0}I_v(n)/I_v(n+1)$ is not Noetherian. 
Then $A_{\mathfrak m}[\mathcal J]_{(zt)}$ and $A_{\mathfrak m}[\mathcal J]_{(zt)}/\mathfrak m [\mathcal J]_{(zt)}$ are not  Noetherian rings. Thus the schemes $\mbox{Proj}(A_{\mathfrak m}[\mathcal J])$ and $\mbox{Proj}(A_{\mathfrak m}[\mathcal J]/\mathfrak m
A_{\mathfrak m}[\mathcal J])$
are not Noetherian.
\end{Proposition}

\begin{proof} Let $T=A[\mathcal I]_{(zt)}$, the elements of degree 0 in the localization $A[\mathcal I]_{zt}$.
$$
\begin{array}{lll}
T&=&\sum_{n\ge 0}(\sum_{i+j=n}\frac {I_v(i)z^j}{z^n}A)\\
&=&A+\frac{I_w(1)}{z}+\frac{I_w(2)}{z^2}+\cdots+\frac{I_w(n)}{z^n}+\cdots\\
&=&A+\frac{I_v(1)}{z}+\frac{I_v(2)}{z^2}+\cdots+\frac{I_v(n)}{z^n}+\cdots\subset A[\frac{1}{z}],
\end{array}
$$
 and so,
$$
T=\cdots\oplus \frac{I_v(n)}{z^n}\oplus\frac{I_v(n-1)}{z^{n-1}}\oplus\cdots\oplus\frac{I_v(1)}{z}\oplus R\oplus zR\oplus z^2R\oplus\cdots
$$
since $I_v(i+1)\subset I_v(i)$ for all $i$.  Thus $T$  is a graded $R$-algebra, graded by $\deg_z$. 

For $n>0$, 
$$
z\left(\frac{I_v(n)}{z^n}\right)=\frac{I_v(n)}{z^{n-1}}\mbox{ and }m_R\frac{I_v(n)}{z^n}=\frac{m_RI_v(n)}{z^n}
\subset \frac{I_v(n+1)}{z^n}.
$$
For $n\ge 0$, 
$$
z(z^nR)=z^{n+1}R\mbox{ and }m_Rz^nR=z^nm_RR.
$$
Thus
$$
T/\mathfrak m T\cong 
\cdots+\frac{1}{z^n}I_v(n)/I_v(n+1)+\cdots+\frac{1}{z}I_v(1)/I_v(2)+R/m_R
\cong \oplus_{n\ge 0}I_v(n)/I_v(n+1).
$$
$A=R\oplus zR\oplus z^2R\oplus \cdots\subset T$ so $A(T/\mathfrak m T)=R/m_R$. Thus 
$$
(A\setminus \mathfrak m)(T/\mathfrak m T)=\{f\in R/m_R\mid f\ne 0\}
$$
are the units in $T/\mathfrak m T$ and so
$(T/\mathfrak m T)\otimes_A A_{\mathfrak m}\cong T/\mathfrak m T$, and thus 
$$
A_{\mathfrak m}[\mathcal J]_{(zt)}/\mathfrak mA_{\mathfrak m} [\mathcal J]_{(zt)}     \cong T/\mathfrak m T
$$
 is not Noetherian, and therefore
 $$
 A_{\mathfrak m}[\mathcal J]_{(zt)}\cong T\otimes_AA_{\mathfrak m }
 $$
  is not Noetherian. $\mbox{Spec}( A_{\mathfrak m}[\mathcal J]_{(zt)})$ is a non Noetherian   affine open subset of 
  $\mbox{Proj}(A_{\mathfrak m}[\mathcal J])$ and $\mbox{Spec}(A_{\mathfrak m}[\mathcal J]_{(zt)}/{\mathfrak m}A_{\mathfrak m}[\mathcal J]_{(zt)})$  is a non Noetherian affine open subset of $\mbox{Proj}(A_{\mathfrak m}[\mathcal J]/\mathfrak m A_{\mathfrak m}[\mathcal J])$.  
  Thus $\mbox{Proj}(A_{\mathfrak m}[\mathcal J])$ and $\mbox{Proj}(A_{\mathfrak m}[\mathcal J]/\mathfrak m A_{\mathfrak m}[\mathcal J])$
  are non Noetherian.  
    \end{proof}

The following theorem follows from Proposition \ref{PropEx1}. 
\begin{Theorem} \label{Theorem2} There exist three dimensional normal excellent local rings $S$  and divisorial filtrations of $m_S$-primary ideals in $S$ such that $\mbox{Proj}(S[\mathcal I])$ is not Noetherian. 
\end{Theorem}

\section{an invariant}\label{SecInv} 
In this section, we extend some material in Section 10 of \cite{C1}. The proofs are the same. Let $R$ be a domain with quotient field $K$ and $\mathcal I=\{I_n\}$ be a graded family of ideals in $R$. 
Let $v$ be a discrete valuation of $K$ which is nonnegative on $R$. The facts that
$I_m^n\subset I_{mn}$ and $v(I_m^n)=nv(I_m)$ imply
$$
\frac{v(I_{mn})}{mn}\le \frac{nv(I_m)}{mn}=\frac{v(I_m)}{m},
$$
and so 
\begin{equation}\label{eqIn2}
\frac{v(I_{mn})}{mn}\le
\min\{\frac{v(I_m)}{m},\frac{v(I_n)}{n}\}.
\end{equation}
We define
$$
\gamma_v(\mathcal I)=\inf\{\frac{v(I_n)}{n}\}.
$$
where the infimum is over all positive $n$. 

By (\ref{eqIn2}), if $\frac{v(I_m)}{m}=\gamma_v(\mathcal I)$ for some $m>0$, then
$$
\frac{v(I_{mn})}{mn}=\gamma_v(\mathcal I)\mbox{ for all $n>0$.}
$$

\begin{Proposition}\label{Prop6} Suppose that $v$ is a valuation of $K$ which is nonnegative on  $R$. Then $v$ has a center on 
$Z(\mathcal I)=\mbox{Proj}(R[\mathcal I])$ if and only if there exists a positive integer $m$ such that 
$$
\frac{v(I_m)}{m}=\gamma_v(\mathcal I).
$$
\end{Proposition}

\begin{proof} First suppose that there does not exist $m$ such that $\frac{v(I_m)}{m}=\gamma_v(\mathcal I)$ but $v$ does have a center on $Z(\mathcal I)$. Then there exists a homogeneous prime ideal $P$ in $R[\mathcal I]$ such that $\oplus_{n>0}I_n\not \subset P$ and $R[\mathcal I]_{(P)}$ is dominated by $\mathcal O_v$.  Thus, there exists $F\in I_m$ for some $m$ such that $0\ne Ft^m\not\in P$, and $R[\mathcal I]_{(Ft^m)}\subset R[\mathcal I]_{(P)}\subset \mathcal O_v$. By assumption, there exists $n$ such that $
\frac{v(I_n)}{n}<\frac{v(I_m)}{m}$. So $\frac{v(I_{mn})}{mn}<\frac{v(I_m)}{m}$ by (\ref{eqIn2}). So there exists an  element $H\in I_{mn}$ such that 
$\frac{v(H)}{mn}=\frac{v(I_{mn})}{mn}<\frac{v(I_m)}{m}$. Since $v(F)\ge v(I_m)$, this implies that $v(\frac{H}{F^n})<0$, so that $\frac{H}{F^n}\not\in \mathcal O_v$. But then
$$
\frac{H}{F^n}=\frac{Ht^{mn}}{(Ft^m)^n}\in R[\mathcal I]_{(Ft^m)}
$$
is not in $\mathcal O_v$, a contradiction. Thus if there does not exist $m$ such that $\frac{v(I_m)}{m}=\gamma_v(\mathcal I)$ then $v$ doesn't have a center on $Z(\mathcal I)$.

Now suppose that there exists $m_0$ such that $\frac{v(I_{m_0})}{m_0}=\gamma_v(\mathcal I)$. Then by (\ref{eqIn2}),
$$
\frac{v(I_{nm_0})}{nm_0}\le \frac{v(I_{m_0})}{m_0}=\gamma_v(\mathcal I)
$$
implies
$$
\frac{v(I_{nm_0})}{nm_0}=\gamma_v(\mathcal I)
$$
for $n\ge 1$.
There exists $F\in I_{m_0}$ such that $v(F)=v(I_{m_0})=m_0\gamma_v(\mathcal I)$. Suppose that $x\in R[\mathcal I]_{(Ft^{m_0})}$. Then for some $n>0$
and $G\in I_{m_0n}$,
$$
x=\frac{Gt^{nm_0}}{(Ft^{m_0})^n}=\frac{G}{F^n}.
$$
 So
$$
\frac{v(x)}{nm_0}=\frac{v(G)-nv(F)}{nm_0}\ge \frac{v(I_{nm_0})-nv(F)}{nm_0}=0.
$$
Thus $x\in \mathcal O_v$ and so $R[\mathcal I]_{(Ft^{m_0})}\subset \mathcal O_v$. Let $q=m_{v}\cap R[\mathcal I]_{(Ft^{m_0})}$. 
Then $q\in \mbox{Spec}(R[\mathcal I]_{(Ft^{m_0})})\subset Z(\mathcal I)$ and $\mathcal O_{Z(\mathcal I),q}=(R[\mathcal I]_{(Ft^{m_0})})_q$ is a local ring of $Z(\mathcal I)$ which is dominated by $v$.
Thus if there exists $m_0$ such that $\frac{v(I_{m_0})}{m_0}=\gamma_v(\mathcal I)$ then $v$ has a center on $Z(\mathcal I)$.
\end{proof}

\section{A graded filtration whose fiber cone has infinitely many irreducible components}

If $\dim R=2$ (and is normal) and $\mathcal I=\{I_n\}$ is a divisorial filtration of $R$, then $\mbox{Proj}(R[\mathcal I])$ is Noetherian (by \cite{R} and \cite{C1}), so that the closed subscheme
$\mbox{Proj}(R[\mathcal I]/m_RR[\mathcal I])$ is Noetherian. Thus the fiber cone $R[\mathcal I]/m_RR[\mathcal I]$ has only finitely many associated primes.

Let $R$ be a 2 dimensional regular local ring with algebraically closed residue field. 
 Let $X_0=\mbox{Spec}(R)$ and $X_1$ be the blowup of the maximal ideal $m_R$ of $R$, with integral exceptional divisor $E(1)_1$. Let $p_1\in X_1$ be a closed point on $E(1)_1$. Let $X_2\rightarrow X_1$ be the blowup of $p_1$. Let $E(2)_1$ be the strict transform of $E(1)_1$ on $X_2$ and $E(2)_2$ be the integral exceptional divisor mapping to $p_1$.
Let $X_3\rightarrow X_2$ be the blowup of a closed point $p_2\in E(2)_2\setminus E(2)_1$. Let $E(3)_3$ be the integral exceptional divisor mapping to $p_2$. Let $E(3)_1$ be the strict transform of $E(2)_1$ and $E(3)_2$ be the strict transform of $E(2)_2$.
Inductively continue this construction, letting $X_l\rightarrow X_{l-1}$ be the blowup of a point $p_{l-1}\in E(l-1)_{l-1}\setminus E(l-1)_{l-2}$. Let $E(l)_l$ be the integral exceptional divisor mapping to $p_{l-1}$ and $E(l)_i$ be the strict transform of $E(l-1)_i$ for $i<l$. Then $X_l$ is nonsingular, and the reduced exceptional divisor of $X_l\rightarrow \mbox{Spec}(R)$ is the divisor
$E(l)_1+E(l)_2+\cdots+E(l)_l$. 
When $l$ is understood,  we will often write $E_i$ for $E(l)_i$.  For $l>1$, we have the intersection products
$$
(E_1\cdot E_j)=\left\{\begin{array}{ll}
-2&\mbox{ if }j=1\\
1&\mbox{ if }j=2\\
0&\mbox{ otherwise.}
\end{array}\right.
$$
For $1<i<l$,
$$
(E_i\cdot E_j)=\left\{\begin{array}{ll}
1&\mbox{ if $j=i-1$ or $j=i+1$}\\
-2&\mbox{ if }j=i\\
0&\mbox{ otherwise,}
\end{array}\right.
$$
$$
(E_l\cdot E_j)=\left\{\begin{array}{ll}
1&\mbox{ if }j=l-1\\
-1&\mbox{ if }j=l\\
0&\mbox{ otherwise.}
\end{array}\right.
$$
These formulas follow from the fact that the intersection graph of $X_l$ is a linear graph with nodes $E(l)_1,E(l)_2,\ldots, E(l)_l$ and weights  (self intersection numbers) $-2,-2,\ldots,-2,-1$ for $l\ge 2$. We now establish that the intersection graph has this form. We have that $((E(1)_1)^2))=-1$. Let $\phi_2:X_2\rightarrow X_1$ be the blowup of $p_1$. Then
$$
(E(2)_1\cdot(E(2)_1+E(2)_1))=
(E(2)_1\cdot \phi_2^*(E(1)_1))
=((\phi_2)_*(E(2)_1)\cdot E(1)_1)=(E(1)_1\cdot E(1)_1)
$$
by the projection formula. thus
$$
(E(2)_1\cdot E(2)_1)=-1-(E(2)_1\cdot E(2)_2)=-2
$$
establishing the form of the intersection graph for  $l=2$. By induction on $l$, we now find that the graph of $X_l$ has the desired form for all $l$, using the projection formula for $X_l\rightarrow X_{l-1}$, and the fact that $p_{l-1}$ only lies on $E(l-1)_{l-1}$ and not on $E(l-1)_{i}$ for $i<l-1$.

For $l\ge 1$, define a $\QQ$-divisor $F(l)$ on $X_l$ by 
$$
\begin{array}{lll}
F(l)&=&E_1+\frac{3}{2}E_2+\frac{7}{4}E_3+\cdots +\frac{2^{n-1}-1}{2^{n-2}}E_{n-1}+\frac{2^n-1}{2^{n-1}}E_n+\frac{2^{n+1}-1}{2^n}E_{n+1}\\
&&+\cdots +\frac{2^{l-1}-1}{2^{l-2}}E_{l-1}+\frac{2^l-1}{2^{l-1}}E_l.
\end{array}
$$
For integers $m\ge 1$, 
define integral divisors $D(l)_m$ on $X_l$ by
$D(l)_m=\lceil mF(l)\rceil$. If $l$ is understood, we will often denote $D(l)_m$ by $D_m$.

\subsection{Calculation of a family of divisors}

We now compute the intersection numbers $(-D_m\cdot E_n)$ for $l\ge 2$, $m\ge 1$ and $1\le n\le l$.

\noindent{\bf Case 1.} Assume that $n=1$. Then $(-D_m\cdot E_1)=2m-\lceil \frac{3}{2}m\rceil$. Write $m=2x+r$ with $x\in \NN$ and $0\le r<2$. 

\noindent{\bf Case 1.1} Assume $r=0$. Then $\lceil \frac{3}{2}m\rceil =3x$ so that 
$$
(-D_m\cdot E_1)=4x-3x=x.
$$

\noindent{\bf Case 1.2.} Assume $r=1$. Then $\lceil \frac{3}{2}m\rceil =3x+\lceil \frac{3}{2}\rceil=3x+2$ and
$$
(-D_m\cdot E_1)=4x+2-(3x+2)=x.
$$

\noindent{\bf Case 2. $1<n<l$.}
 Then
$$
\begin{array}{lll}
(-D_m\cdot E_n)&=&-\lceil \frac{m(2^{n-1}-1)}{2^{n-2}}\rceil+2\lceil \frac{m(2^n-1)}{2^{n-1}}\rceil-\lceil\frac{m(2^{n+1}-1)}{2^n}\rceil\\
&=& -2m+\lfloor \frac{m}{2^{n-2}}\rfloor+4m-2\lfloor\frac{m}{2^{n-1}}\rfloor-2m+\lfloor \frac{m}{2^n}\rfloor\\
&=& \lfloor \frac{m}{2^{n-2}}\rfloor-2\lfloor \frac{m}{2^{n-1}}\rfloor+\lfloor \frac{m}{2^n}\rfloor.
\end{array}
$$
Write $m=2^nx+r$ with $0\le r<2^n$.

\noindent {\bf Case 2.1: $0\le r<2^{n-2}$. } Then
$$
\lfloor \frac{m}{2^{n-2}}\rfloor =4x+\lfloor \frac{r}{2^{n-2}}\rfloor=4x,
\lfloor \frac{m}{2^{n-1}}\rfloor = 2x+\lfloor \frac{r}{2^{n-1}}\rfloor =2x,
\lfloor \frac{m}{2^n}\rfloor=x+\lfloor \frac{r}{2^n}\rfloor =x.
$$
Thus
$$
(-D_m\cdot E_n)=\lfloor \frac{m}{2^{n-2}}\rfloor -2\lfloor \frac{m}{2^{n-1}}\rfloor +\lfloor \frac{m}{2^n}\rfloor
=4x-4x+x=x.
$$

\noindent{\bf Case 2.2: $2^{n-2}\le r<2^{n-1}$.} Write $r=2^{n-2}+s$ with $0\le s<2^{n-2}$ so that 
$$
\frac{s}{2^{n-1}}<\frac{1}{2}.
$$
We have that
$$
\lfloor \frac{m}{2^{n-2}}\rfloor =\lfloor \frac{2^nx+2^{n-2}+s}{2^{n-2}}\rfloor 
= 4x+1+\lfloor \frac{s}{2^{n-2}}\rfloor =4x+1,
$$
$$
\lfloor \frac{m}{2^{n-1}}\rfloor =\lfloor \frac{2^nx+r}{2^{n-1}}\rfloor =2x+\lfloor \frac{r}{2^{n-1}}\rfloor =2x,
$$
$$
\lfloor \frac{m}{2^n}\rfloor =  \lfloor \frac{2^nx+r}{2^n}\rfloor=    x+\lfloor \frac{r}{2^n}\rfloor =x. 
$$
Thus
$$
(-D_m\cdot E_n)=\lfloor \frac{m}{2^{n-2}}\rfloor -2\lfloor \frac{m}{2^{n-1}}\rfloor+\lfloor\frac{m}{2^n}\rfloor
=4x+1-4x+x=x+1.
$$

\noindent{\bf Case 2.3: $m=2^nx+r$, $2^{n-1}\le r<2^n$.} This implies that $r=2^{n-1}+s$ with $0\le s<2^{n-1}$.

\noindent{\bf Case 2.3.1: $2^{n-2}\le s<2^{n-1}$.} Thus $s=2^{n-2}+t$ with $0\le t<2^{n-2}$. We have
$$
\lfloor \frac{m}{2^{n-2}}\rfloor = \lfloor \frac{2^nx+2^{n-1}+2^{n-2}+t}{2^{n-2}}\rfloor=4x+3+\lfloor \frac{t}{2^{n-2}}\rfloor=4x+3,
$$
$$
\lfloor \frac{m}{2^{n-1}}\rfloor =\lfloor \frac{2^nx+2^{n-1}+s}{2^{n-1}}\rfloor=2x+1+\lfloor \frac{s}{2^{n-1}}\rfloor=2x+1,
$$
$$
\lfloor \frac{m}{2^n}\rfloor = \lfloor \frac{2^nx+r}{2^n}\rfloor=x+\lfloor\frac{r}{2^n}\rfloor =x.
$$
$$
(-D_m\cdot E_n)=\lfloor\frac{m}{2^{n-2}}\rfloor-2\lfloor\frac{m}{2^{n-1}}\rfloor+\lfloor \frac{m}{2^n}\rfloor
=4x+3-2(2x+1)+x=x+1.
$$

{\bf Case 2.3.2: $0\le s<2^{n-2}$.} Then we have
$$
\lfloor \frac{m}{2^{n-2}}\rfloor=\lfloor \frac{2^nx+2^{n-1}+s}{2^{n-2}}\rfloor=4x+2+\lfloor \frac{s}{2^{n-2}}\rfloor=4x+2,
$$
$$
\lfloor \frac{m}{2^{n-1}}\rfloor =\lfloor \frac{2^nx+2^{n-1}+s}{2^{n-1}}\rfloor=2x+1+\lfloor \frac{s}{2^{n-1}}\rfloor =2x+1,
$$
$$
\lfloor \frac{m}{2^n}\rfloor =\lfloor \frac{2^nx+r}{2^n}\rfloor=x+\lfloor \frac{r}{2^n}\rfloor =x.
$$
Thus
$$
(-D_m\cdot E_n)=\lfloor \frac{m}{2^{n-2}}\rfloor -2\lfloor\frac{m}{2^{n-1}}\rfloor+\lfloor \frac{m}{2^n}\rfloor=4x+2-2(2x+1)+x=x.
$$

\noindent{\bf Case 3: $n=l$.} 
$$
\begin{array}{lll}
(-D_m\cdot E_l)&=&-\lceil\frac{m(2^{l-1}-1)}{2^{l-2}}\rceil+\lceil \frac{m(2^l-1)}{2^{l-1}}\rceil\\
&=& -2m+\lfloor \frac{m}{2^{l-2}}\rfloor+2m -\lfloor \frac{m}{2^{l-1}}\rfloor\\
&=& \lfloor \frac{m}{2^{l-2}}\rfloor -\lfloor \frac{m}{2^{l-1}}\rfloor.
\end{array}
$$
Write $m=2^{l-1}x+r$ with $0\le r<2^{l-1}$.

\noindent{\bf Case 3.1: $0\le r<2^{l-2}$.} Then
$$
\lfloor \frac{m}{2^{l-2}}\rfloor =2x+\lfloor \frac{r}{2^{l-2}}\rfloor =2x,
$$
$$
\lfloor\frac{m}{2^{l-1}}\rfloor =x+\lfloor\frac{r}{2^{l-1}}\rfloor=x.
$$
Thus
$$
(-D_m\cdot E_l)=2x-x=x.
$$

\noindent{\bf Case 3.2: $2^{l-2}\le r<2^{l-1}$.} Thus $r=2^{l-2}+s$ with $0\le s<2^{l-2}$. We have
$$
\lfloor \frac{m}{2^{l-2}}\rfloor =2x+1+\lfloor \frac{s}{2^{l-2}}\rfloor=2x+1,
$$
$$
\lfloor\frac{m}{2^{l-1}}\rfloor =x+\lfloor\frac{r}{2^{l-1}}\rfloor=x.
$$
Thus
$$
(-D_m\cdot E_l)=2x+1-x=x+1.
$$

Let $\pi:W\rightarrow \mbox{Spec}(R)$ be a proper morphism and $G$ be a Cartier divisor on $W$. Then $G$ is said to be nef if $(C\cdot G)\ge 0$ for all closed integral curves $C\subset \pi^{-1}(m_R)$.

\begin{Proposition}\label{Prop1} Let the $X_l$ and divisors $D(l)_m$ be as defined above. Then
\begin{enumerate}
\item[1)] $-D(l)_m$ is nef on $X_l$ for all $l$ and $m$.
\item[2)] Let $\pi_{b,a}:X_b\rightarrow X_a$ be the natural $R$-morphism for $b\ge a$. Then
$$
\pi_{b,a}^*(D(a)_m)=D(b)_m\mbox{ if }m<2^{a-1}.
$$
\item[3)] Suppose that $1\le n\le l$. Then 
$$
(-D(l)_m\cdot E(l)_n)>0 \mbox{ if $m>2^n$.}
$$
\end{enumerate}
\end{Proposition}

\begin{proof} Statement 1) follows from the above calculations. To prove 2), we observe that the calculation of  Case 3.1 shows that $(-D(a+1)_m\cdot E_{a+1})=0$ so that $\pi_{a+1,a}^*(D(a)_m)=D(a+1)_m$.
The conclusions of 2) thus follow from induction on $b$. The conclusions of 3) follow from the cases of the above calculations.
\end{proof}

Since $R$ is a two dimensional regular local ring, $\mbox{Spec}(R)$ has a rational singularity, as commented after the definition Definition 1.1 \cite{L} of rational singularity  and so $H^1(X_l,\mathcal O_{X_l})=0$ by Proposition 1.2 \cite{L}. 
Since $-D(l)_m$ is nef, we have that the sheaf
$\mathcal O_{X_l}(-D(l)_m)$ is generated by global sections for all $l,m\ge 1$ by Theorem 12.1 \cite{L}. Further, all powers of the $m_R$-primary ideal 
$\Gamma(X_l,\mathcal O_{X_l}(-D(l)_m)$ in $R$ are integrally closed since $R$ is a two dimensional regular local ring by Theorem 2, Appendix 5 \cite{ZS2}, Theorem 7.1 \cite{L} or Theorem 3.7 \cite{Hu}. Thus 
$Y_{l,m}=\mbox{Proj}(R[\Gamma(X_l,\mathcal O_{X_l}(-D(l)_m)t])$ is normal and there is a natural morphism $\tau_{l,m}:X_l\rightarrow Y_{l,m}$ for all $m\ge 1$, which contracts the $E(l)_i$ which have intersection number zero with $-D(l)_m$.

There is a natural divisorial valuation $v_{E_n}$ of $R$ whose center on $X_l$ is $E(l)_n$ and whose center on $R$ is $m_R$ for $n\le l$. The valuation ring $\mathcal O_{v_{E_n}}$ is the local ring $\mathcal O_{X_l,E(l)_n}$ for all $l\ge n$. Further, $E(l)_n$ is not contracted on $Y_{l,m}$ for $n\le l$ and $m>2^{n}$ by 3) of Proposition \ref{Prop1}. Since $Y_{l,m}$ is normal, 
\begin{equation}\label{eq9}
\mbox{the center $\tau_{l,m}(E(l)_n)$ of $v_{E_n}$ on $Y_{l,m}$ is a prime divisor}
\end{equation}
 and the local ring $\mathcal O_{Y_{l,m},\tau_{l,m}(E(l)_n)}$ is the valuation ring $\mathcal O_{v_{E_n}}$.

\subsection{Construction of the example}\label{SubSecCon}
Define a graded filtration of $m_R$-primary ideals $\mathcal I=\{I_m\}$ by
$$
I_m=\{f\in R\mid v_{E_i}(f)\ge \frac{m(2^i-1)}{2^{i-1}}\mbox{ for }i\ge 1\}.
$$
Let 
$$
J_{l,m}=\{f\in R\mid v_{E_i}(f)\ge \frac{m(2^i-1)}{2^{i-1}}\mbox{ for }1\le i\le l\}
$$
We have that
$$
J_{l,m}=\Gamma(X_l,\mathcal O_{X_l}(-D(l)_m)).
$$
 By 2) of Proposition \ref{Prop1}, for $l'>l$,
$J_{l',m}=J_{l,m}$ if $m<2^{l-1}$. Thus
\begin{equation}\label{eq8}
I_m=\cap_{a\ge 1}J_{a,m}=J_{l,m}
\end{equation}
for all $l$ such that  $m<2^{l-1}$.

We have that
$$
\begin{array}{lll}
v_{E_n}(I_m)&=&v_{E_n}(\Gamma(X_l,\mathcal O_{X_l}(-D(l)_m))\mbox{ for $m<2^{l-1}$ and $n\le l$}\\
&=& v_{E_n}(\Gamma(X_l,\mathcal O_{X_l}(-D(l)_m))\mathcal O_{v_{E_n}})=\lceil \frac{m(2^n-1)}{2^{n-1}}\rceil,
\end{array}
$$
since $\mathcal O_{X_i,E(l)_n}=\mathcal O_{v_{E_n}}$ and $\mathcal O_{X_l}(-D(l)_m)$ is generated by global sections.
Thus
\begin{equation}\label{eq5}
v_{E_n}(I_m)=\lceil \frac{m(2^n-1)}{2^{n-1}}\rceil.
\end{equation}
for all $m$ and $n$.

Let $n>1$ and choose $q_n\in E_n(n+1)\setminus (E_{n-1}(n+1)\cup E_{n+1}(n+1))$. Then $\pi_{l,n+1}:X_l\rightarrow X_{n+1}$ is an isomorphism above a neighborhood of $q_n$ for all $l>n+1$, so we may identify $q_n$ with the preimage of $q_n$ in $X_l$. Thus $\mathcal O_{X_n,q_n}=\mathcal O_{X_{n+1},q_n}$ for all $l>n+1$.
Let $v$ be the valuation which is the composite of $v_{E_n}$ and the $m$-adic valuation of $q_n$ on $E_n$. 
We construct this valuation explicitly. A general treatment of composite valuations is given in Section 10 of \cite{Ab} and Chapter VI, Section 10 \cite{ZS2}, but knowledge of the general theory is not necessary to understand the explicit construction given here. 
Suppose that $x,y$ are regular parameters in $\mathcal O_{X_{n+1},q_n}$, where $x=0$ is a local equation of $E_n$.  Let $0\ne f\in \mathcal O_{X_{n+1},q_n}$. Write $f=x^ag$ where $g\in\mathcal O_{X_{n+1},q_n}$ and $x\not{|}  g$. Let $\overline{g}$ be the residue of $g$ in the one dimensional regular local ring $\mathcal O_{E_n,q_n}$. Let $b=\mbox{ord}(\overline g)$. 
Then $v(f)=(a,b)\in (\ZZ^2)_{\rm lex}$.

Now $I_m\mathcal O_{X_l,q_n}=J_{l,m}\mathcal O_{X_l,q_n}$ for $l>\max\{n+1,{\rm log}_2m+1\}$ by (\ref{eq8}).
Thus
\begin{equation}\label{eq6}
v(I_m)=(\lceil \frac{m(2^n-1)}{2^{n-1}}\rceil,0)
\end{equation}
since $\mathcal O_{X_l}(-D_m)$ is generated by global sections.
With the notation of Section \ref{SecInv},
\begin{equation}\label{eq7}
\gamma_v(\mathcal I)=\inf\{\frac{v(I_m)}{m}\}=
(\frac{2^n-1}{2^{n-1}},0).
\end{equation}

\begin{Proposition}\label{Prop2}
For all $n\ge 1$, $v$ and $v_{E_n}$ have a center on $\mbox{Proj}(R[\mathcal I])$.
\end{Proposition}
\begin{proof} Since the valuation ring $\mathcal O_{v_{E_n}}$ is a localization of the valuation ring $\mathcal O_v$, it suffices to show that $v$ has a center on $\mbox{Proj}(R[\mathcal I])$.
If $2^{n-1}$ divides $m$, then we see from (\ref{eq6}) and (\ref{eq7}) that
$$
\frac{v(I_m)}{m}=(\frac{2^n-1}{2^{n-1}},0)=\gamma_{v}(\mathcal I).
$$
Thus $v$ has a center on $\mbox{Proj}(R[\mathcal I])$ by Proposition \ref{Prop6}.

\end{proof}

Let $Z=\mbox{Proj}(R[\mathcal I])$ with natural projection $\lambda:Z\rightarrow \mbox{Spec}(R)$.

Let $\alpha,\beta$ be the respective centers of $v$ and $v_{E_n}$ on $Z$. Let $P_n\subset Q_n$ be the corresponding homogeneous prime ideals in $R[\mathcal I]=\sum_{i\ge 0}I_it^i$. 
Let $\mathfrak p\subset \mathfrak q$ be the respective prime ideals of $\beta$ and $\alpha$ in $\mathcal O_{Z,\alpha}=R[\mathcal I]_{(Q_n)}$. The valuation ring $\mathcal O_v$ dominates $\mathcal O_{Z,\alpha}$. The prime ideals of $\mathcal O_v$ are $0\subset \mathfrak n\subset \mathfrak m$ where $\mathfrak m$ is the maximal ideal of $\mathcal O_v$ and the localization $(\mathcal O_v)_{\mathfrak n}=\mathcal O_{v_{E_n}}$. We have that $\mathfrak m\cap \mathcal O_{Z,\alpha}=\mathfrak q$ and $\mathfrak n\cap \mathcal O_{Z,\alpha}=\mathfrak p$.


For $2^n<m<2^{l-1}$, $I_m=\Gamma(X_l,\mathcal O_{X_l}(-D_m))$ by (\ref{eq8}) and  the centers of $v$ and $v_{E_n}$ on $\mbox{Proj}(R[I_mt^m])=Y_{l,m}$ are respectively a closed point $\gamma$ and a curve $\delta$ by 3) of Proposition \ref{Prop1}. 

The respective corresponding homogeneous prime ideals in $R[I_mt^m]$ are $p_n=P_n\cap R[I_mt^m]$ and $q_n=Q_n\cap R[I_mt^m]$.
The image $\lambda(\beta)=m_R$ since the center of $v_{E_n}$ on $R$ is $m_R$, so $P_n\cap R=m_R$ and $m_RR[\mathcal I]\subset P_n$.
Since $p_n\subset q_n$ are distinct prime ideals we have that 
$$
P_n\subset Q_n\subset m_RR[\mathcal I]+\sum_{i>0}I_it^i
$$ 
are distinct prime ideals. Thus $\dim R[\mathcal I]/P_n\ge 2$. 
 We have that $\dim R[\mathcal I]/P_n\le \dim R[\mathcal I]/m_RR[\mathcal I]=\ell(\mathcal I)$, the analytic spread of $\mathcal I$. We  have that $\ell(\mathcal I)\le \dim R=2$ by Lemma 3.6 \cite{CS}. Thus $\dim R[\mathcal I]/P_n=2$ for all $n$.

By (\ref{eq9}), for $a<b$, $v_{E_a}$ and $v_{E_b}$ have distinct centers on $Y_{l,m}$ if $2^b<m<2^{l-1}$. Thus the prime ideals $p_a=P_a\cap R[I_mt^m]$ and $p_b=P_b\cap R[I_mt^m]$ are distinct. Thus the prime ideals $P_a$ and $P_b$ are distinct. Consequently, 
the one dimensional scheme $\mbox{Proj}(R[\mathcal I]/m_RR[\mathcal I])$ has infinitely many distinct irreducible components of dimension one. 

Finally, we show that $Z$ is not Noetherian. 
 Suppose that $Z=\mbox{Proj}(R[\mathcal I])$ is Noetherian. Then the underlying topological space of $Z$ is Noetherian. This implies that every nonempty closed subset $Y$ of $Z$ can be expressed as a finite union of irreducible closed subsets of $Z$ by \cite[Proposition I.1.5]{H}. But the one dimensional subscheme $Y=\mbox{Proj}(R[\mathcal I]/m_RR[\mathcal I])$ has infinitely many one dimensional irreducible components. This contradiction shows that $Z$ is not noetherian.

Thus we have completed the proof of the following theorem.

\begin{Theorem}\label{Theorem3} There exists a graded filtration $\mathcal I$ of $m_R$-primary ideals on a two dimensional regular local ring $R$ such that 
the ideal  $m_RR[\mathcal I]$ has infinitely many distinct minimal primes, and the one dimensional scheme $\mbox{Proj}(R[\mathcal I]/m_RR[\mathcal I])$ has infinitely many distinct one dimensional irreducible components.
In particular, the two dimensional scheme  $\mbox{Proj}(R[\mathcal I])$ is not  Noetherian.
 \end{Theorem}

\end{document}